\documentclass{amsart}

\newcommand{\IR}{\mathbb R}
\newcommand{\IB}{\mathbb B}
\newcommand{\IN}{\mathbb N}

\newcommand{\e}{\varepsilon}
\newcommand{\Conv}{\mathrm{Conv}}
\newcommand{\HH}{\mathcal H}
\newcommand{\dens}{\mathrm{dens}}
\newcommand{\Ra}{\Rightarrow}
\newcommand{\conv}{\mathrm{conv}}
\newcommand{\cone}{\mathrm{cone}}
\newcommand{\w}{\omega}
\newcommand{\dist}{\mathsf{dist}}
\newcommand{\cl}{\mathrm{cl}}
\newcommand{\aff}{\mathrm{aff}}
\newcommand{\lin}{\mathrm{lin}}
\newcommand{\F}{\mathcal F}

\newcommand{\dH}{\mathsf{d}_{\mathsf H}\kern-1pt}
\newcommand{\ConvH}{\Conv_{\mathsf H}}

\newtheorem{theorem}{Theorem}[section]
\newtheorem{lemma}[theorem]{Lemma}
\newtheorem{claim}[theorem]{Claim}
\theoremstyle{definition}
\newtheorem{remark}{Remark}

\title[On approximatively polyhedral convex sets]{A ``hidden'' characterization of\\ approximatively polyhedral convex sets in Banach spaces}
\author{Taras Banakh and Ivan Hetman}
\address{T.Banakh: Ivan Franko National University of Lviv (Ukraine), and Universytet Jana Kochanowskiego, Kielce (Poland)}
\email{t.o.banakh@gmail.com}
\address{I.Hetman: Department of Mathematics, Ivan Franko National University of Lviv, Universytetska 1, 79000, Ukraine}
\email{ihromant@gmail.com}

\begin{document}
\begin{abstract} For a Banach space $X$ by $\ConvH(X)$ we denote the space of non-empty closed convex subsets of $X$, endowed with the Hausdorff metric. We prove that for any closed convex set $C\subset X$ and its metric component $\HH_C=\{A\in\ConvH(X):\dH(A,C)<\infty\}$ in $\ConvH(X)$, the following conditions are equivalent:
\begin{enumerate}
\item $C$ is approximatively polyhedral, which means that for every $\e>0$ there is a polyhedral convex subset $P\subset X$ on Hausdorff distance $\dH(P,C)<\e$ from $C$;
\item $C$ lies on finite Hausdorff distance $\dH(C,P)$ from some polyhedral convex set $P\subset X$;
\item the metric space $(\HH_C,\dH)$ is separable;
\item $\HH_C$ has density $\dens(\HH_C)<\mathfrak c$;
\item $\HH_C$ does not contain a positively hiding convex set $P\subset X$.
\end{enumerate}
If the Banach space $X$ is finite-dimensional, then the conditions (1)--(5) are equivalent to:
\begin{enumerate}
\item[(6)] $C$ is not positively hiding;
\item[(7)] $C$ is not infinitely hiding.
\end{enumerate}
A convex subset $C\subset X$ is called {\em positively hiding} (resp. {\em infinitely hiding}) if there is an infinite set $A\subset X\setminus C$ such that $\inf\limits_{a\in A}d(a,C)>0$ (resp. $\sup\limits_{a\in A}d(a,C)=\infty$) and for any distinct points $a,b\in A$ the segment $[a,b]$ meets the set $C$.
\end{abstract}
\maketitle

\section{Introduction}

In \cite{BH} the authors proved that a closed convex subset $C$ of a complete linear metric space $X$ is polyhedral in its linear hull if and only if $C$ hides no infinite subset $A\subset X\setminus C$ in the sense that $[a,b]\cap C\ne\emptyset$ for any distinct points $a,b\in A$. In this paper we shall prove a similar ``hidden'' characterization of approximatively polyhedral subsets in Banach spaces, simultaneously giving a characterization of separable components of the space of closed convex subsets $\ConvH(X)$ of a Banach space $X$.

For a Banach space $X$ by $\ConvH(X)$ we denote the space of all non-empty closed convex subsets of $X$,  endowed with the {\em Hausdorff metric}
$$\dH(A,B)=\max\{\sup_{a\in A}\dist(a,B),\sup_{b\in B}\dist(b,A)\}\in[0,\infty].$$
Here $\dist(a,B)=\inf\limits_{b\in B}\|a-b\|$ stands for the distance from a point $a\in X$ to a subset $B\subset X$ of the Banach space $X$.

It is well-known  that for each closed convex set $C\in\ConvH(X)$ the Hausdorff distance $\dH$ restricted to the set
$$\HH_C=\{A\in\ConvH(X):\dH(A,C)<\infty\}$$is a metric, see \cite[Ch2]{Sendov}. The obtained metric space $(\HH_C,\dH)$ will be called the {\em Hausdorff metric component} (or just {\em component}) of the set $C$ in $\ConvH(X)$.

In fact, the present investigation was motivated by the problem of calculating the density of components of the space $\ConvH(X)$ and detecting closed convex subsets $C\subset X$ with separable component $\HH_C$. In this paper we shall characterize such sets $C$ in terms of approximative polyhedrality as well as in ``hidden'' terms resembling those from \cite{BH}.

A convex subset $C$ of a Banach space $X$ is called
\begin{itemize}
\item {\em a closed half-space} if $C=f^{-1}([a,+\infty)$ for  some non-zero linear continuous functional $f:X\to\IR$ and some real number $a\in\IR$;
\item {\em polyhedral} if $C$ can be written as the intersection $C=\cap \mathcal F$ of a finite family $\mathcal F$ of closed half-spaces in $X$;
\item {\em approximatively polyhedral} if for every $\e>0$ there is a closed polyhedral subset $P\subset X$ that lies on  Hausdorff distance $\dH(C,P)<\e$ from $C$.
\end{itemize}

Observe that the whole space $X$ is polyhedral, being the intersection $X=\cap\F$ of the empty family $\F=\emptyset$ of closed half-spaces.

It is well-known that each compact convex subset of Banach space is approximatively polyhedral (see \cite{Gruber} for more information on that topic). This is not necessarily true for non-compact closed convex sets. For example, the convex parabola
$$P=\big\{(x,y)\in\IR^2:y\ge x^2\big\}$$is not approximatively polyhedral in $\IR^2$, while the convex hyperbola
$$H=\big\{(x,y)\in\IR^2:y\ge \sqrt{x^2+1}\big\}$$is approximatively polyhedral.

Next, we introduce some ``hidden'' properties of convex sets. Following \cite{BH}, we say that a subset $C$ of a linear space $X$ {\em hides} a set $A\subset X$ if for any two distinct points $a,b\in A$ the affine segment $[a,b]=\{ta+(1-t)b:t\in[0,1]\}$ meets the set $C$.

A convex subset $C$ of a Banach space $X$ is called
\begin{itemize}
\item {\em hiding} if $C$ hides some infinite set $A\subset X\setminus C$;
\item {\em positively hiding} if $C$ hides  some infinite set $A\subset X\setminus C$ such that $\inf_{a\in A}\dist(a,C)>0$;
\item {\em infinitely hiding}  if $C$ hides  some infinite set $A\subset X\setminus C$ such that $\sup_{a\in A}\dist(a,C)=\infty$.
\end{itemize}

It is clear that each infinitely hiding set is positively hiding and each positively hiding set is hiding.

By \cite{BH}, a closed convex subset $C$ of a complete linear metric space $X$ is hiding if and only if $C$ is not polyhedral in its closed linear hull. So, both parabola and hyperbola are hiding (being not polyhedral). Yet, the parabola is infinitely hiding (but not approximatively polyhedral) while the hyperbola is not positively hiding (but is approximatively polyhedral).

It turns out that the approximative polyhedrality and positive or infinite hiding properties are mutually exclusive, and can be characterized via properties of the characteristic cone of a given convex set.

Let us recall that the {\em characteristic cone} of a convex subset $C$ in a linear topological space $X$ is the set $V_C$ of all vectors $v\in X$ such that for every point $c\in C$ the ray $c+\bar \IR_+v=\{c+tv:t\ge 0\}$ lies in $C$. Here $\bar \IR_+=[0,\infty)$ denotes the closed half-line. The characteristic cone $V_C$ is closed in $X$ if $C$ is closed or open in $X$, see Lemma~\ref{l2.2n}.

The main result of this paper is the following characterization theorem that will be used in the paper \cite{BHS} devoted to recognizing the topological structure of the space $\ConvH(X)$. In finite-dimensional case, the equivalence of the conditions (1)--(3) was proved by Viktor Klee in \cite{Klee}.

\begin{theorem}\label{main} For a closed convex subset $C$ of a Banach space $X$ the following conditions are equivalent:
\begin{enumerate}
\item $C$ is approximatively polyhedral;
\item the characteristic cone $V_C$ of $C$ is polyhedral in $X$ and $\dH(C,V_C)<\infty$;
\item the component $\HH_C$ contains a polyhedral closed convex set;
\item the component $\HH_C$ contains no positively hiding closed convex set;
\item the space $\HH_C$ is separable;
\item the space $\HH_C$ has density $\dens(\HH_C)<\mathfrak c$.
\end{enumerate}
If the Banach space $X$ is finite-dimensional, then the conditions (1)--(6) are equivalent to:
\begin{enumerate}
\item[(7)] $C$ is not positively hiding;
\item[(8)] $C$ is not infinitely hiding.
\end{enumerate}
\end{theorem}

Let us recall that the {\em density} $\dens(X)$ of a topological space $X$ is the smallest cardinality $|D|$ of a dense subset $D$ of $X$. Topological spaces with at most countable density are called {\em separable}.

\begin{remark} Observe that the closed unit ball $C=\{x\in l_2:\|x\|\le 1\}$ in the separable Hilbert space $l_2$ is positively hiding but not infinitely hiding. This example shows that the conditions $(7)$ and $(8)$ are not equivalent in infinite-dimensional Banach spaces.
\end{remark}

Theorem~\ref{main} will be proved in Section~\ref{pf:main} after long preliminary work made in Sections~\ref{s2}--\ref{s6}.

\section{Some properties of characteristic cones}\label{s2}

This section is of preliminary character and contains some information on convex cones in Banach spaces. All linear (and Banach) spaces considered in this paper are over the field of real numbers $\IR$.

By a {\em convex cone} in a linear space $X$ we understand a convex subset $C\subset X$ such that $tc\in C$ for any  $t\in\bar\IR_+$ and $c\in C$. Here $\bar \IR_+=[0,+\infty)$ stands for the closure of the open half-line $\IR_+=(0,+\infty)$ in $\IR$.
For two subsets $A,B$ of a linear space $X$ and a real number $\lambda$, let $A+B=\{a+b:a\in A,\;b\in B\}$ be the pointwise sum of $A$ and $B$,  and  $\lambda A=\{\lambda\cdot a:a\in A\}$ be a homothetic copy of $A$. %{\em All convex cones and all convex subsets considered in this paper are not empty.}

Each subset $F\subset X$ generates the cone
$$\cone(F)=\Big\{\sum_{i=1}^n\lambda_ix_i:\mbox{$n\in\IN$ and $(x_i)_{i=1}^n\in F^n$, $(\lambda_i)_{i=1}^n\in\bar\IR_+^n$}\Big\},$$which contains the convex hull $\conv(F)$ of $F$.

The following description of polyhedral cones and polyhedral convex sets in finite-dimensional spaces is classical and  can be found in \cite{Klee}, \cite[Theorems 1.2, 1.3]{Ziegler} or \cite[\S4.3]{Gallier}:

\begin{lemma}\label{l2.1} Let $X$ be a finite-dimensional Banach space.
\begin{enumerate}
\item A convex cone $C\subset X$ is polyhedral if and only if $C=\cone(F)$ for some finite set $F\subset X$.
\item A convex set $C\subset X$ is polyhedral if and only if $C=\cone(F)+\conv(E)$ for some finite sets $F,E\subset X$.
\end{enumerate}
\end{lemma}

We shall be mainly interested in characteristic cones and dual characteristic cones of convex sets in Banach spaces. Let us recall that for a convex subset $C$ of a Banach space $X$ its {\em characteristic cone} $V_C$ is defined by
$$V_C=\{x\in X:\forall c\in C\;\;c+\bar\IR_+x\subset C\}\subset X.$$
By the {\em dual characteristic cone} of $C$ we understand the convex cone
$$V_C^*=\{x^*\in X^*:\sup x^*(C)<\infty\}$$lying in the dual Banach space $X^*$.

It is clear that the dual characteristic cone $V_C^*$ of a convex set $C\subset X$ coincides with the dual characteristic cone $V_{\bar C}^*$ of its closure $\bar C$ in $X$. The relation between the characteristic cones $V_C$ and $V_{\bar C}$ are described in the following simple lemma, whose proof is left to the reader as an exercise.

\begin{lemma}\label{l2.2n} Let $\bar C$ be the closure of a convex set $C$ in a Banach space $X$. Then
\begin{enumerate}
\item $V_C\subset V_{\bar C}$;
\item $V_C=V_{\bar C}$ if the set $C$ is open in $X$.
\end{enumerate}
\end{lemma}

Our next aim is to show that the characteristic cones of two closed convex subsets $A,B\subset X$ with  $\dH(A,B)<\infty$ coincide. For this we shall need:

\begin{lemma}\label{l2.3} For each point $c_0$ of a convex set $C$ in a Banach space $X$, each vector $v\notin V_{\bar C}$ and each real number $\e\in\IR_+$ there is a real number $t\in\IR_+$ such that $\dist(c_0+tv,C)=\e$.
\end{lemma}

\begin{proof} Since $v\notin V_{\bar C}$, there is a real number $t_0>0$ such that $c_0+t_0v\notin \bar C$. Consider the continuous function $$f:\bar\IR_+\to\bar\IR_+,\;\;f:t\mapsto \dist(c_0+tv,C),$$and observe that $f(0)=0$ as $c_0\in C$. We claim that $\lim_{t\to+\infty}f(t)=+\infty$. Since $c_0+t_0v\notin \bar C$, we can apply Hahn-Banach Theorem and find a linear functional $x^*\in X^*$ with unit norm such that $x^*(c_0+t_0v)>\sup x^*(\bar C)\ge x^*(c_0)$, which implies that $x^*(v)>0$. Then for any real number $t>t_0$ we get $$\dist(t v,C)=\inf_{c\in C}\|c_0+tv-c\|\ge \inf_{c\in C} |x^*(tv)-x^*(c-c_0)|=t x^*(v)-\sup x^*(C-c_0)$$ and hence $\lim_{t\to+\infty}\dist(c_0+tv,C)=\infty$. By the continuity of $f$, there is a positive real number $t$ with $\dist(c_0+tv,C)=f(t)=\e$.
\end{proof}

Now we can prove the promised:

\begin{lemma}\label{l2.4n} Let $A,B$ be two closed convex sets in a Banach space $X$. If $\dH(A,B)<\infty$, then $V_A=V_B$.
\end{lemma}

\begin{proof} We lose no generality assuming that $0\in A\cap B$. If $V_A\ne V_B$, then we can find a vector $v\in X$ that lies in $V_A\setminus V_B$ or in $V_B\setminus V_A$. We lose no generality assuming that $v\in V_B\setminus V_A$. By Lemma~\ref{l2.3}, there is a positive real number $t$ such that $\dist(tv,A)>\dH(A,B)$, which is not possible as $tv\in V_B\subset B$.
\end{proof}

Observe that for a convex $C$ containing zero, the inclusion $V_C\subset C$ implies $V_C^*\subset V^*_{V_C}$.

\begin{lemma}\label{l2.5n} For any closed convex set $C$ in a Banach space the dual characteristic cone $V^*_{V_C}$ of the characteristic cone $V_C$ of $C$ coincides with the weak$^*$ closure $\cl^*(V_C)$ of $V_C$.
\end{lemma}

\begin{proof} We lose no generality assuming that $0\in C$. Observe that the dual characteristic cone $$V^*_{V_C}=\{x^*\in X^*:\sup x^*(V_C)<\infty\}=\{x^*\in X^*:\sup x^*(V_C)=0\}=\bigcap_{v\in V_C}\{x^*\in X^*:x^*(v)\le 0\}$$is weak$^*$ closed in $X^*$, being an intersection of weak$^*$-closed half-spaces in $X^*$. So, the inclusion $V^*_C\subset V^*_{V_C}$ implies $\cl^*(V^*_C)\subset V^*_{V_C}$. It remains to prove the reverse inclusion $V^*_{V_C}\subset \cl^*(V^*_C)$. Assume conversely that it is not true and find a functional $x^*\in V^*_{V_C}\setminus\cl^*(V^*_C)$. By the Hahn-Banach Theorem applied to the weak$^*$ topology of $X^*$, there is an element $x\in X$ that separates $x^*$ from $\cl^*(V^*_C)$ in the sense that $$x^*(x)>\sup\{v^*(x):v^*\in\cl^*(V^*_C)\}\ge\sup\{v^*(x):v^*\in V^*_C\}.$$ We claim that $v^*(x)\le 0$ for all functionals $v^*\in V^*_C$. Assuming that $v^*(x)>0$, we can find a positive real number $\lambda$ so large that $\lambda v^*(x)>x^*(x)$, which contradicts the choice of $x$ (because $\lambda v^*\in V^*_C$). So, $v^*(x)\le 0$ for all $v^*\in V^*_C$. We claim that $x\in V_C$.

In the opposite case, we could find a positive real number $t>0$ such that $tx\notin C$ (here we recall that $0\in C$). Applying Hahn-Banach Theorem, find a linear functional $v^*\in X^*$ such that  $v^*(tx)>\sup v^*(C)\ge 0$. Then $v^*\in V^*_C$ and $v^*(x)>0$, which contradicts the property of $x$ established at the end of the preceding paragraph. This contradiction shows that $x\in V_C$ and then $x^*(x)>0$ implies that $\sup x^*(V_C)=\infty$, which contradicts the choice of the functional $x^*\in V^*_{V_C}$. This contradiction completes the proof of the inclusion $V^*_{V_C}\subset \cl^*(V^*_C)$.
\end{proof}

The following lemma implies that polyhedral convex sets in Banach spaces lie on positive Hausdorff distance from their characteristic cones.

\begin{lemma}\label{l2.6} For a normed space $X$,  functionals  $f_1,\dots,f_n:X\to\IR$, a  vector $\mathbf a=(a_1,\dots,a_n)\in\IR^n$ with non-negative coordinates, and the polyhedral convex set $$P_{\mathbf a}=\bigcap\limits_{i=1}^nf_i^{-1}\big((-\infty,a_i]\big)$$
\begin{enumerate}
\item  the characteristic cone $V_{P_{\mathbf a}}$ of the convex set $P_{\mathbf a}$ coincides with the polyhedral cone $P_{\mathbf 0}$;
\item $\dH(P_{\mathbf a},P_{\mathbf 0})\le \dH(P_{\mathbf 0},P_{\mathbf 1})\cdot\max\limits_{1\le i\le n}a_i$,
\end{enumerate} where $\mathbf 0=(0,\dots,0)$ and  $\mathbf 1=(1,\dots,1)$.
\end{lemma}

\begin{proof} We consider the space $\IR^n$ as a Banach lattice with coordinatewise operations of minimum and maximum.
\smallskip

1. The first statement is easy and is left to the reader as an exercise.
\smallskip

2. To prove the second statement, we first check that $\dH(P_{\mathbf 0},P_{\mathbf 1})<\infty$.
By Lemma~\ref{l2.1}(2), $P_{\mathbf 1}=\conv(F)+\cone(E)$ for some finite sets $F,E\subset X$.
It follows that the cone $\cone(E)$ coincides with the characteristic cone $P_{\mathbf 0}$ of $P_{\mathbf 1}$ and hence $P_{\mathbf 1}=\conv(F)+P_{\mathbf 0}$. Then
$$\dH(P_{\mathbf 1},P_{\mathbf 0})\le\dH(\conv(F)+P_{\mathbf 0},P_{\mathbf 0})\le\dH(\conv(F),\{0\})<\infty.$$
Let $a=\max\limits_{1\le i\le n}a_i$. Taking into account that that the norm of the Banach space $X$ is homogeneous and that $P_{\mathbf 0}\subset P_{\mathbf a}\subset P_{a\mathbf 1}$, we get the required inequality
$$\dH(P_{\mathbf a},P_{\mathbf 0})\le\dH(P_{a\mathbf 1},P_{\mathbf 0})=a\cdot\dH(P_{\mathbf 1},P_{\mathbf 0})=\dH(P_{\mathbf 1},P_{\mathbf 0})\cdot \max_{1\le i\le n}a_i<\infty.$$
\end{proof}

\section{Recognizing separable components of $\ConvH(X)$}

In this section we shall prove some lemmas that will help us to recognize closed convex sets $C$ with separable component $\HH_C$. First we consider the finite-dimensional case. The following lemma was proved by V.Klee in \cite{Klee}. However, we give an alternative proof based on a Ramsey-theoretic argument.

\begin{lemma}\label{l3.1} If the component $\HH_C$ of a closed convex subset $C$ of a finite-dimensional Banach space $X$ contains a polyhedral convex set, then $\HH_C$ contains a countable dense family of polyhedral convex sets.
\end{lemma}

\begin{proof} The case $C=X$ is trivial because in this case the component $\HH_C=\{X\}$ contains a unique convex set $X$, which is polyhedral as the intersection of the empty family of closed half-spaces. So, we assume that $\HH_C$ contains some polyhedral convex set $P\ne X$. Without loss of generality we can assume that $0\in P$. By Lemma~\ref{l2.4n}, $\dH(C,P)<\infty$ implies $V_C=V_P\ne X$.

 Write $P$ as a finite intersection of closed half-spaces
$$P=\bigcap_{i=1}^kf_i^{-1}\big((-\infty,a_i]\big)$$
where $f_1,\dots,f_k:X\to\IR$ are linear continuous functionals with unit norm and $a_1,\dots,a_k$ are some real numbers. It follows from $0\in P$ that $a_1,\dots,a_n\ge 0$. According to Lemma~\ref{l2.6}, we lose no generality assuming that $a_1=\dots=a_k=0$, which implies that $P$ is a polyhedral cone that coincides with its characteristic cone $V_P=V_C$. By Lemma~\ref{l2.1}, $P=\cone(B)$ for some finite subset $B\subset X$.

By our assumption, the Banach space $X$ is finite-dimensional and hence separable. So, we can fix a countable dense subset $D\subset X$. Next, for every finite subset $F\subset D$ consider its convex hull $\conv(F)$ and the polyhedral convex set $$C_F=\conv(F)+V_C=\conv(F)+\cone(B).$$ It remains to check that the countable family $$\mathcal C=\{C_F:\mbox{$F$ is a finite subset of $D$}\}$$
is dense in $\HH_C$.

Given a convex set $A\in\HH$ and $\e>0$, we shall find a finite subset $F\subset D$ with $\dH(C_F,A)< 2\e$. Denote by
$\bar \IB=\{x\in X:\|x\|\le 1\}$ the closed unit ball of the Banach space $X$. Then for every $r>0$ the set $r\bar\IB=\{r\cdot x:x\in\bar\IB\}$ coincides with the closed $r$-ball $\{x\in X:\|x\|\le r\}$.

\begin{claim}\label{cl3.2} There is $r\in\IR_+$ so large that the convex set $A_r=(A\cap r\bar\IB)+P$ is not empty and lies on the Hausdorff
distance $\dH(A_r,A)\le\e$ from $A$.
\end{claim}

\begin{proof} It follows from $\dH(A,C)<\infty$ that $V_A=V_C=P$, see Lemma~\ref{l2.4n}. Then for each $r\in\IR_+$ we get the inclusion
$$A_r=(A\cap r\bar\IB)+P\subset A+P=A+V_A=A.$$
Assuming that $\dH(A_r,A)>\e$ for all $r\in\IR_+$, we can construct an increasing sequence of positive real numbers $(r_n)_{n\in\w}$ and
a sequence of points $(x_n)_{n\in\w}$ in $A$ such that $\|x_n\|\le r_n$ and $\dist(x_{n+1},A_{r_{n}})>\e$ for all $n\in\w$. Consequently,
for every $n<m$ we get $$(x_m+\e\bar\IB)\cap(x_n+P)\subset (x_m+\e\bar\IB)\cap (A_{r_n}+P)=\emptyset,$$
which implies $x_m-x_n\notin\e\bar\IB+P$.

Let us recall that $P=\bigcap_{i=1}^k H_i$ where $H_i=f_i^{-1}\big((-\infty,0]\big)$ for $i\le k$. Using Lemma~\ref{l2.6}, we can choose $\delta>0$ such that
$\bigcap_{i=1}^kf_i^{-1}\big((-\infty,\delta]\big)\subset P+\e\IB$.

It follows that for any $n<m$ we get $x_m-x_n\notin \bigcap_{i=1}^nf_i^{-1}\big((-\infty,\delta]\big)$
and hence there is a number $i=i(n,m)\in\{1,\dots,k\}$ such that $x_m-x_n\notin f_i^{-1}\big((-\infty,\delta]\big)$ and hence $f_i(x_m)>f_i(x_n)+\delta$.

The correspondence $i:(n,m)\to i(n,m)$ can be thought as a finite coloring of the set $[\w]^2=\{(n,m)\in\w^2:n<m\}$ of pairs of positive
integers. The Ramsey Theorem 5 of \cite{Ramsey} yields an infinite subset $\Omega\subset \w$ and a number $i\in\{1,\dots,k\}$ such that $i(n,m)=i$ and hence $f_i(x_m)>f_i(x_n)+\delta$ for
all numbers $n<m$ in $\Omega$. This implies $\sup_{c\in C}f_i(c)\ge\sup_{n\in \Omega}f_i(x_n)=\infty$,  which is not possible as $\sup f_i(C)\le (\sup f_i(P))+\dH(C,P)=\dH(C,P)<\infty$.
\end{proof}

Claim~\ref{cl3.2} yields us a real number $r\in\IR_+$ such that the intersection $A\cap r\bar\IB$ is not empty and $\dist(A_r,A)<\e$ where
$A_r=(A\cap r\bar\IB)+P$. By \cite{Gruber}, the compact convex set $A\cap r\bar\IB$ can be approximated by a finite subset $F\subset D$ such that $\dH(\conv(F),A\cap r\bar \IB)<\e$. Then the polyhedral convex set $C_F=\conv(F)+P$ lies on Hausdorff distance $\dH(C_F,A_r)<\e$ from the
set $A_r$ and hence on Hausdorff distance $\dH(C_F,A)\le \dH(C_F,A_r)+\dH(A_r,A)<2\e$ from $A$.
\end{proof}

In order to generalize Lemma~\ref{l3.1} to infinite-dimensional Banach spaces $X$, we now establish some simple properties of maps between spaces of closed convex sets, induced by quotient operators.

Let us recall that for a Banach space $(X,\|\cdot\|_X)$ and a closed linear subspace $Z\subset X$, the quotient Banach space $Y=X/Z$ is endowed with the norm
$$\|y\|_Y=\inf\big\{\|x\|_X:x\in q^{-1}(y)\big\},$$
where $q:X\to Y$, $q:x\mapsto x+Z$, stands for the quotient operator.
%It is well-known that the quotient operator $q$ is open.

The quotient operator $q:X\to Y$ induces an operator $\bar{q}:\ConvH(X)\to\ConvH(Y)$ assigning to each closed convex set $C\subset X$ the closure $\bar qC$ of its image $qC$ in the Banach space $Y$. The following lemma is simple and is left to the reader as an exercise.

\begin{lemma}\label{l3.3} Let $Z$ be a closed linear subspace of a Banach space $X$, $Y=X/Z$ be the quotient Banach space, and $q:X\to Y$ be the quotient operator.
\begin{enumerate}
\item A convex set $C\subset X$ with $Z\subset V_C$ is closed in $X$ if and only if the image $qC$ is closed in $Y$.
\item A convex set $C\subset X$ with $Z\subset V_C$ is polyhedral in $X$ if and only if its image $qC$ is polyhedral in $Y$.
\item For any non-empty convex sets $A,B\subset X$ with $Z\subset V_A\cap V_B$ we get $\dH(A,B)=\dH(qA,qB)$.
\end{enumerate}
\end{lemma}

Now we are able to prove an (infinite-dimensional) generalization of Lemma~\ref{l3.1}, which will be used in the proof of the implications $(3)\Ra(1,5)$ from Theorem~\ref{main}.

\begin{lemma}\label{l3.4} If the component $\HH_C$ of a non-empty closed convex subset $C$ of a Banach space $X$ contains a polyhedral convex set, then $\HH_C$ contains a countable dense family of polyhedral closed sets, which implies that the space $\HH_C$ is separable and the convex set $C$ is approximatively polyhedral.
\end{lemma}

\begin{proof} The statement of the lemma is trivial if $C=X$. So, we assume that $C\ne X$ and $\HH_C$ contains a polyhedral convex set $P$. Replacing $P$ by its shift, we can assume that $0\in P$. Write $P$ as a finite intersection of closed half-spaces
$$P=\bigcap_{i=1}^kf_i^{-1}\big((-\infty,a_i]\big),$$
where $f_1,\dots,f_k:X\to\IR$ are suitable linear continuous functionals and $a_1,\dots,a_k$ are suitable non-negative real numbers. It follows from $\dH(C,P)<\infty$ that the characteristic cone $$V_C=V_P=\bigcap_{i=1}^kf_i^{-1}\big((-\infty,0]\big)$$
is polyhedral and the closed linear subspace
$$Z=-V_C\cap V_C=\bigcap_{i=1}^kf_i^{-1}(0)$$has finite codimension in $X$.

Then the quotient Banach space $Y=X/Z$ is finite-dimensional. Taking into account that $Z\subset V_P\cap V_C$ and applying Lemma~\ref{l3.3}(1,3), we conclude that the images $qC$ and $qP$ are closed convex sets in $Y$ with $\dH(qC,qP)<\infty$. Moreover, the convex set $qP$ is polyhedral in $Y$. Since the Banach space $Y$ is finite-dimensional, it is legal to apply Lemma~\ref{l3.1} in order to find a dense countable subset $\mathcal D_Y\in\HH_{qC}$ that consists of polyhedral convex sets. By Lemma~\ref{l3.3}(2), the countable family $\mathcal D_X=\{q^{-1}(D):D\in\mathcal D_X\}$ consists of polyhedral convex subsets of $X$ and by Lemma~\ref{l3.3}(3) is dense in the component $\HH_C$ of $C$.
%It remains to check that $\mathcal D_X\subset \HH_C$ is dense in $\HH_C$. To show that $\mathcal D_X\subset\HH_C$, fix any set $D\in\mathcal D_Y$ and apply Lemma~\ref{??}(?) to conclude that
%$$\dH(q^{-1}(D),C)=\dH(D,qC)<\infty,$$which implies $q^{-1}(D)\in\HH_C$.
%The density of $\mathcal D_X$ in $\HH_C$ will be established as soon as given a convex set $A\in\HH_C$ and an $\e>0$ we find a polyhedral convex set $D'\in\mathcal D_X$ with $\dH(A,D')<\e$.
%For this consider the image $qA$ of $A$ in $Y$ and applying Lemma~\ref{??}(1,3), observe that it closed in $Y$ and $\dH(qA,qC)=\dH(A,C)<\infty$, which means that $qA\in\HH_{qC}$. The density of the family $\mathcal D_X$ in $\HH_{qC}$ yields a convex set $D\in\mathcal D_X$ with $\dH(qA,D)<\e$. Then the set $d'=q^{-1}(D)\in\mathcal D_X$ has the required property: $\dH(D',A)=\dH(qD',qA)=\dH(D,qA)<\e$ by Lemma~\ref{??}(3).
\end{proof}

\section{Recognizing non-separable components of $\ConvH(X)$}

In this section we develop some tools for recognizing non-separable components of the space $\ConvH(X)$.

\begin{lemma}\label{l4.1} Let $C$ be a convex subset of a linear space $X$ and $a,b\in X$ two points such that $[a,b]\cap C\ne\emptyset$. Then for any points $x\in\conv(C\cup\{a\})$ and $y\in\conv(X\cup\{b\})$ we get $[x,y]\cap C\ne\emptyset$.
\end{lemma}

\begin{proof} The lemma trivially holds if $x$ or $y$ belongs to the set $C$. So, we assume that $x,y\notin C$. It follows that the point $x\in\conv(C\cup\{a\})\setminus C$ can be written as $x=t_xa+(1-t_x)c_x$ for some $t_x\in(0,1]$ and $c_x\in C$. The same is true for the point $y\in\conv(C\cup\{b\})\setminus C$ which can be written as $y=t_yb+(1-t_y)c_y$ for some  $t_y\in(0,1]$ and $c_y\in C$.

By our assumption, the intersection $[a,b]\cap C$ contains some point $c=ta+(1-t)b$ where $t\in[0,1]$.

The lemma will be proved as soon as we check that the segment $[x,y]$ meets the convex hull $\conv(\{c,c_x,c_y\})\subset C$ of the 3-element set $\{c,c_x,c_y\}\subset C$. This will follow as soon as we find real numbers $u,\alpha,\alpha_x,\alpha_y\in[0,1]$ such that $\alpha+\alpha_x+\alpha_y=1$ and
$$\alpha c+\alpha_xc_x+\alpha_yc_y=ux+(1-u)y=u(t_xa+(1-t_x)c_x)+(1-u)(t_yb+(1-t_y)c_y).$$

The number $u\in[0,1]$ can be found from the equation
$$ut_xa+(1-u)t_yb=(ut_x+(1-u)t_y)(ta+(1-t)b)=(ut_x+(1-u)t_y)c,$$
which has a well-defined solution $$u=\frac{t\cdot t_y}{t\cdot t_y+(1-t)t_x}.$$
The remaining numbers $\alpha,\alpha_x$ and $\alpha_y$ can be found as
$$\alpha=ut_x+(1-u)t_y,\quad \alpha_x=u(1-t_x),\quad\alpha_y=(1-u)(1-t_y).$$
\end{proof}

The following lemma will be used for the proof of the implication $(6)\Ra(4)$ from Theorem~\ref{main}.

\begin{lemma}\label{l4.2} The component $\HH_C\subset\ConvH(X)$ of a closed convex subset $C$ of a Banach space $X$ has density
$\dens(\HH_C)\ge\mathfrak c$ provided $\HH_C$ contains a positively hiding closed convex subset $P$ of $X$.
\end{lemma}

\begin{proof} Since $\HH_C=\HH_P$, we lose no generality assuming that the convex set $C$ itself is positively hiding, which means that there is an infinite subset $A\subset X\setminus C$ on positive distance $\e=\inf_{a\in A}\dist(a,C)$ from $C$, which is hidden behind the set $C$ in the sense that for any distinct points $a,b\in A$ the affine segment $[a,b]$ meets the set $C$.

Fix any point $c_0\in C$ and for every point $a\in A$ choose a point $b_a\in[c_0,a]$ on the distance $\dist(b_a,C)=\e$ from $C$. The choice of the point $b_a$ is possible as $\dist(a,C)\ge\e$. Lemma~\ref{l4.1} guarantees that the set $B=\{b_a:a\in A\}$ is infinite and is hidden behind the set $C$. Moreover, this set lies in the $2\e$-neighborhood $C+2\e\IB$ of $C$. Here $\IB=\{x\in X:\|x\|< 1\}$ stands for the open unit ball in the Banach space $X$.

Now for any subset $\beta\subset B$ consider the convex set $C_\beta=\conv(C\cup\beta)$. Applying Lemma~\ref{l4.1} one can show that this set is closed in $X$ and  $C_\beta=\bigcup_{b\in\beta}\conv(C\cup\{b\})$. Taking into account that $C\subset C_\beta\subset C+2\e\IB$, we see that $\dH(C,C_\beta)\le 2\e$ and hence $C_\beta$ belongs to the component $\HH_C$ of $C$.

We claim that for any two distinct subsets $\alpha,\beta\subset B$, the convex sets $C_\alpha$ and $C_\beta$ lie on the Hausdorff distance $\dH(C_\alpha,C_\beta)\ge\e$. Since $\alpha\ne\beta$, there is a point $b\in (\beta\setminus\alpha)\cup(\alpha\setminus\beta)$. We lose no generality assuming that $b\in\beta\setminus\alpha$. Then $b\in C_\beta$ and $\dist(b,C_\alpha)\ge\e$. Indeed, assuming that $\dist(b,C_\alpha)<\e$, we conclude that the open $\e$-ball $b+\e\IB$ meets the set $C_\alpha=\conv(C\cup\alpha)=\bigcup_{a\in \alpha}\conv(C\cup\{a\})$ at some point $x$ that belongs to $\conv(C\cup\{a\})$ for some point $a\in\alpha$. Since the set $B\ni a,b$ is hidden behind the set $C$, the segment $[a,b]$ meets the set $C$. By Lemma~\ref{l4.1}, the segment $[x,b]$ also meets the set $C$, which is not possible as $[x,b]$ lies in the $\e$-ball $b+\e \IB$, which does not meet $C$ as $\dist(b,C)=\e$.
This contradiction shows that $\dist(b,C_\alpha)\ge\e$ and hence $\dH(C_\beta,C_\alpha)\ge\e$.

Now we see that the component $\HH_C$ contains the subset $\mathcal C=\{C_\beta:\beta\subset B\}$ of cardinality $|\mathcal C|\ge 2^{|B|}\ge\mathfrak c$, consisting of points on mutual distance $\ge \e$. This implies that $\dens(\HH_C)\ge |\mathcal C|\ge \mathfrak c$.
\end{proof}

\section{Recognizing infinitely hiding convex sets}

In this section we develop some tools for recognizing infinitely hiding convex sets.
In fact, we shall work with the following relative version of this property.

Let $C_0,C$ be two convex sets in a Banach space $X$. We shall say that $C_0$ is {\em $C$-infinitely hiding} if $C_0$ hides some infinite set $A\subset\aff(C_0)$ such that $\sup_{a\in A}\dist(a,C)=\infty$.

It is easy to see that a convex set $C\subset X$ is infinitely hiding if and only if it is $C$-infinitely hiding.

We start with the following elementary lemma.

\begin{lemma}\label{l5.1} Let $C\ni 0$ be a convex set in a Banach space and $V_{\bar C}$ be the characteristic cone of its closure. For a linear subspace $Z\subset X$, the intersection $Z\cap V_{\bar C}$ is $C$-infinitely hiding if the cone $Z\cap V_{\bar C}$ is a hiding convex set in $Z$.
\end{lemma}

\begin{proof} Assume that the cone $Z\cap V_{\bar C}$ hides some infinite injectively enumerated set $\{a_n\}_{n\in\w}\subset Z\setminus V_{\bar C}$. By Lemma~\ref{l2.3}, for every $n\in\w$, there is a real number $t_n>0$ such that $\dist(t_n a_n,C)>n$. It is clear that for the set $A=\{t_na_n\}_{n\in\w}$ we get $\sup\limits_{a\in A}\dist(a,C)=\lim\limits_{n\to\infty}\dist(t_na_n,C)=\infty$. It remains to show that for every distinct numbers $n,m\in\w$ the segment $[t_na_n,t_mb_m]$ intersects the cone $Z\cap V_{\bar C}$.

Since the set $\{a_n,a_m\}\subset A\subset Z$ is hidden behind $Z\cap V_{\bar C}$, the segment $[a_n,a_m]$ meets the cone $Z\cap V_{\bar C}$ at some point $c=\tau a_n+(1-\tau)a_m$ where $\tau\in[0,1]$. Then for the number
$$u=\frac{\tau t_m}{\tau t_m+(1-\tau)t_n}\in[0,1]$$we get
$$
ut_na_n+(1-u)t_ma_m=\frac{t_nt_n}{\tau t_m+(1-\tau)t_n}(\tau a_n+(1-\tau)a_m)=\frac{t_nt_m}{\tau t_m+(1-\tau)t_n}c\in [t_na_n,t_ma_m]\cap V_{\bar C}
$$and hence the intersection $Z\cap V_{\bar C}\cap[t_na_n,t_ma_m]\ni ut_na_n+(1-u)t_ma_m$ is not empty.
\end{proof}

By \cite{BH}, a closed convex subset $C$ of a complete linear metric space $X$ is hiding if and only if $C$ is polyhedral in its closed linear hull. This characterization combined with Lemma~\ref{l5.1} implies:

\begin{lemma}\label{l5.2} Let $C\ni 0$ be a convex set in a Banach space and $V_{\bar C}$ be the characteristic cone of its closure. For a closed linear subspace $Z\subset X$ the intersection $V=Z\cap V_{\bar C}$ is infinitely $C$-hiding if the cone $V$ is polyhedral in its closed linear hull $V^\pm=\cl(V-V)$.
\end{lemma}

This lemma implies its absolute version:

\begin{lemma}\label{l5.3} A closed convex subset $C$ of a Banach space is infinitely hiding if its characteristic cone $V_C$ is not polyhedral in its closed linear hull $V_C^\pm=\cl(V_C-V_C)$.
\end{lemma}

Next, we derive the infinite hiding property of a convex set from that property of its projections. We start with the following algebraic fact.

\begin{lemma}\label{l5.4} Let $q:X\to \tilde X$ be a linear operator between linear spaces, $E=q^{-1}(0)$ be the kernel of $q$,  and $C\subset X$ be a convex set such that $V_{C\cap E}-V_{C\cap E}=E$. If the image $\tilde C=q(C)$ hides some countable set $\tilde A\subset \aff(\tilde C)$, then $C$ hides some set $A\subset \aff(C)$ with $q(A)=\tilde A$.
\end{lemma}

\begin{proof} Let $\tilde A=\{\tilde a_n:n\in\w\}$ be an injective enumeration of the countable set $\tilde A$. By induction, for every $n\in\w$ we shall choose a point $a_n\in q^{-1}(\tilde a_n)\cap\aff(C)$ so that $[a_n,a_m]\cap C\ne\emptyset$ for every numbers $n<m$, and $a_n\in C$ if $\tilde a_n\in \tilde C$.

We start the inductive construction choosing any point $a_0\in q^{-1}(\tilde a_0)\cap\aff(C)$. Such point $a_0$ exists since $q(\aff(C))=\aff(\tilde C)$. If $\tilde a_0\in \tilde C$, then we can additionally assume that $a_0\in C$.  Assume that for some $n\ge 1$ the points $a_0,\dots,a_{n-1}$ have been constructed. We need to choose a point $a_n\in q^{-1}(\tilde a_n)\cap\aff(C)$ so that $[a_i,a_n]\cap C\ne\emptyset$ for all $i<n$. If $\tilde a_n\in\tilde C$, then let $a_n\in C$ be any point with $q(a_n)=\tilde a_n$. So, assume that $\tilde a_n\notin \tilde C$.
Let $I_n=\{i\in \w:i<n,\;\tilde a_i\notin\tilde C\}$.

Since the set $\tilde A\subset\tilde X$ is hidden behind the convex set $\tilde C$, for every $i\in I_n$ the intersection $[\tilde a_i,\tilde a_n]\cap \tilde C$ contains some point $\tilde c_i$ which can be written as the convex combination $\tilde c_i=u_i\tilde a_i+(1-u_i)\tilde a_n$ for some $u_i\in(0,1)$. Since $\tilde c_i\in \tilde C$, there is a point $c_i\in C$ with $q(c_i)=\tilde c_i$. It follows from $\tilde a_n=\frac{\tilde c_i-u_i\tilde a_i}{1-u_i}$ that the point $a_i'=\frac{c_i-u_ia_i}{1-u_i}$ belongs to the preimage $q^{-1}(\tilde a_n)$.

It follows from $E=V_{E\cap C}-V_{E\cap C}$ that the intersection $\bigcap_{i\in I_n}(a_i'+V_{E\cap C})$ contains some point $a_n$. For this point we get $u_ia_i+(1-u_i)a_n\in c_i+V_{C}\subset C$ and hence $[a_i,a_n]\cap C\ne\emptyset$ for all $i<n$, which completes the inductive step.

After completing the inductive construction, we obtain the countable set $A=\{a_n\}_{n\in\w}$ that has the required property.
\end{proof}

Lemma~\ref{l5.4} implies its $C$-infinitely hiding version.

\begin{lemma}\label{l5.5} Let $X$ be a Banach space, $E$ be a closed linear subspace of $X$, $\tilde X=X/E$ be the quotient Banach space and $q:X\to \tilde X$ be the quotient operator. Let $C_0,C$ be two convex sets in $X$ and $\tilde C_0=q(C_0)$, $\tilde C=q(C)$ be their quotient images in $\tilde X$.
The convex set $C_0$ is $C$-infinitely hiding $X$ if its image $\tilde C_0$ is $\tilde C$-infinitely hiding and $E=V_{E\cap C_0}-V_{E\cap C_0}$.
\end{lemma}

\begin{proof} If $\tilde C_0$ is $\tilde C$-infinitely hiding, then $\tilde C_0$ hides some infinite set $\tilde A\subset \aff(\tilde C_0)$ such that $\sup_{\tilde a\in \tilde A}\dist(\tilde a,\tilde C)=\infty$. By Lemma~\ref{l5.4}, there is a set $A\subset \aff(C_0)$ with $q(A)=\tilde A$, hidden behind the set $C$.

Since the quotient operator $q:X\to \tilde X$ has norm $\|q\|\le 1$, for every point $a\in A$ and its image $\tilde a=q(a)$,  we get $\dist(\tilde a,\tilde C)\le \dist(a,C)$. Consequently, $\sup_{a\in A}\dist(a,C)\ge \sup_{\tilde a\in \tilde A}\dist(\tilde a,\tilde C)=\infty$, which means that the set  $C_0$ is $C$-infinitely hiding.
\end{proof}

The preceding lemma allows us to derive the $C$-infinite hiding property of a convex set from that property of its projection. Our next lemma will help us to do the same using the $C$-infinite hiding property of two-dimensional sections of the convex set.

\begin{lemma}\label{l5.6} Let $C$ be a closed convex subset of a Banach space $X$ and $Z$ be a two-dimensional linear subspace of $X$ such that the convex set $C\cap Z$ has non-empty interior $C_0$ in $Z$, which contains zero. If $\dH(C_0,V_{C_0})=\infty$, then the convex set $C_0$ is $C$-infinitely hiding.
\end{lemma}

\begin{proof} The equality $\dH(C_0,V_{C_0})=\infty$ implies that
the open convex subset $C_0$ of the plane $Z$ is not bounded. Consequently, its characteristic cone $V_{C_0}=V_{\bar C_0}=Z\cap V_C$ is unbounded too. Moreover, the cone $V_{C_0}$ is not a plane, not a half-plane, and not a line (otherwise $C_0$ would be on finite Hausdorff distance from its characteristic cone $V_{C_0}$). Consequently, we can choose two linearly independent vectors $e_1,e_2\in Z$ such that the cone $V_{C_0}$ is equal to $\cone(\{e_1\})$ or to  $\cone(\{e_1,e_2\})$. Let $e_1^*,e_2^*\in Z^*$ be the coordinate functionals corresponding to the base $e_1,e_2$ of $Z$. This means that $z=e_1^*(z)e_1+e_2^*(z)e_2$ for each vector $z\in Z$.

If $V_{C_0}=\cone(\{e_1,e_2\})$, then the equality $\dH(C_0,V_{C_0})=\infty$ implies that $\inf e_1^*(C_0)=-\infty$ or $\inf e_2^*(C_0)=-\infty$. We lose no generality assuming that $\inf e_2^*(C_0)=-\infty$.

If $V_{C_0}=\cone(\{e_1\})$, then the equality $\dH(C_0,V_{C_0})=\infty$ implies that $\inf e_2^*(C_0)=-\infty$ or $\sup e_2^*(C_0)=+\infty$. Changing $e_2$ to $-e_2$, if necessary, we can assume that $\inf e_2^*(C_0)=-\infty$.

So, in both cases we can assume that $\inf e_2^*(C_0)=-\infty$.

By induction, we shall construct a sequence of points $(a_n)_{n\in\w}$ in $\cone(\{e_1,-e_2\})$ such that for every $n\in\w$ the following conditions are satisfied:
\begin{enumerate}
\item $\dist(a_n,C)>n$;
\item $e_1^*(a_{n})>e_1^*(a_{n-1})>0$, $e_2^*(a_n)<e_2^*(a_{n-1})<0$, and $\dfrac{|e_2^*(a_n)|}{e_1^*(a_n)}<\dfrac{|e_2^*(a_{n-1})|}{e_1^*(a_{n-1})}$;
\item $[a_n,a_{k}]\cap C_0\ne\emptyset$ for all $k<n$.
\end{enumerate}

We start the inductive construction selecting a point $a_0\in\IR_+\cdot(e_1-e_2)$ on the distance $\dist(a_0,C)>0$ from the set $C$. Such point $a_0$ exists because $e_1-e_2\notin V_{C_0}=Z\cap V_C$. Now assume that for some $n\in\IN$ we have constructed points $a_0,\dots,a_n\in\cone(\{e_1,-e_2\})$ satisfying the conditions (1)--(3). It follows from $\inf e_2^*(C_0)=-\infty$ and $e_1\in V_{C_0}\setminus(-V_{C_0})$ that there exists a point $c\in C_0$ such that $$e^*_1(c)>c_1^*(a_n),\;\;e^*_2(c)<e^*_2(a_n)\mbox{ \ and \ }\frac{|e^*_2(c)|}{e^*_1(c)}<\frac{|e^*_2(a_n)|}{e^*_1(a_n)}.$$
Now consider the vector $v=c-a_n$ and observe that $v\notin V_{C_0}=V_{\bar C_0}$. Consequently, $c+\IR_+v\not\subset \bar C_0$, which allows us to find a point $a_{n+1}\in c+\IR_+ v$ with $\dist(a_{n+1},C)>n+1$ (using Lemma~\ref{l2.3}). It can be shown that the point $a_{n+1}$ satisfies the condition (2). Since the segment $[a_n,a_{n+1}]$ contains the point $c$, it  meets the set $C\cap Z$.

It remains to check that $[a_k,a_{n+1}]\cap C_0\ne\emptyset$ for every number $k<n$. By the inductive assumption, the segment $[a_k,a_n]$ meets the set $C_0$ at some point $c'$.

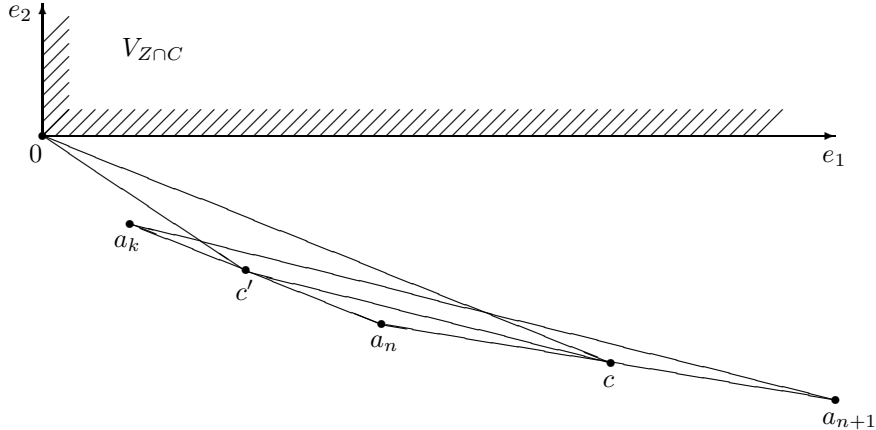
\begin{picture}(300,190)(-30,25)
\put(50,150){\circle*{3}}
\put(50,150){\vector(1,0){300}}
\put(345,140){$e_1$}
\put(50,150){\vector(0,1){50}}
\put(37,195){$e_2$}
\multiput(50,150)(5,0){55}{\line(1,1){10}}
\multiput(50,150)(0,5){8}{\line(1,1){10}}
\put(80,180){$V_{Z\cap C}$}
\put(45,140){$0$}

\put(50,150){\line(3,-2){77}}
\put(127,99){\circle*{3}}
\put(123,88){$c'$}
%\qbezier(127,99)(20,150)(10,200)
\put(127,99){\line(5,-2){51}}
\put(127,99){\line(-5,2){44}}
\put(83,116.8){\circle*{3}}
\put(77,108){$a_k$}
\put(178.3,78.7){\circle*{3}}
\put(174,70){$a_n$}

%\qbezier(127,99)(182,78)(265,64)
%\qbezier(265,64)(320,56)(370,55)
%\put(360,65){$C$}

\put(50,150){\line(5,-2){215}}
\put(265,64){\circle*{3}}
\put(262,55){$c$}
\put(265,64){\line(-4,1){138}}
\put(350,50){\line(-4,1){268}}
\put(350,50){\line(-6,1){172}}
\put(350,50){\circle*{3}}
\put(345,41){$a_{n+1}$}

\end{picture}

 By an elementary plane geometry argument one can prove that the segment $[a_k,a_{n+1}]$ meets the triangle $\conv(\{0,c',c\})\subset C_0$ and hence meets the set $C_0$.
This completes the inductive step.

After completing the inductive construction, we obtain the infinite set $A=\{a_n\}_{n\in\w}$ with $\sup_{a\in A}\dist(a,C)=\infty$ which is hidden behind $C_0$. This means that the set $C_0$ is $C$-infinitely hiding.
\end{proof}

\begin{lemma}\label{l5.7}
Let $C$ be a convex subset of a Banach space $X$ and $Z$ be a finite-dimensional linear subspace of $X$ such that the convex set $C\cap Z$ has non-empty interior $C_0$ in $Z$ and $0\in C_0$. If $\dH(C_0,V_{C_0})=\infty$, then the convex set $C_0$ is $C$-infinitely hiding.
\end{lemma}

\begin{proof} This lemma will be proved by induction on the dimension $\dim(Z)$ of $Z$. The lemma is trivially true if $\dim(Z)\le 1$. Assume that the lemma has been proved for all  triples $(X,C,Z)$ with $\dim(Z)<n$. Now we prove this lemma for $\dim(Z)=n$. Assuming that $\dH(C_0,V_{C_0})=\infty$, we need to prove that the convex set $C_0$ is $C$-infinitely hiding. To derive a contradiction, assume that $C_0$ is not $C$-infinitely hiding. In this case Lemma~\ref{l5.6} implies the following fact, which will be used several times in the subsequent proof.

\begin{claim}\label{cl5.8} For each two-dimensional linear subspace $Z_2\subset Z$ we get $\dH(Z_2\cap C_0,V_{Z_2\cap C_0})<\infty$.
\end{claim}

Now consider the characteristic cone $V_{C_0}=V_{\bar C_0}$ of the open convex set $C_0$ in $Z$.

\begin{claim} The linear subspace $-V_{C_0}\cap V_{C_0}$ is trivial.
\end{claim}

\begin{proof} Assume that the linear subspace $E=-V_{C_0}\cap V_{C_0}$ is not equal to $\{0\}$. Consider the quotient Banach space $\tilde X=X/E$, the quotient operator $q:X\to \tilde X$, the convex set $\tilde C=q(C)$, and the finite-dimensional subspace $\tilde Z=q(Z)$ of dimension $\dim(\tilde Z)<\dim(Z)=n$. Since the quotient operator $q|Z:Z\to\tilde Z$ is open, the convex set $\tilde C_0=q(C_0)$ is open in $\tilde Z$ and hence the convex set $\tilde Z\cap\tilde C$ has non-empty interior in $\tilde Z$. Since $E\subset V_C$, the set $\tilde C$ is closed in the Banach space $\tilde X$ by Lemma~\ref{l3.3}(1). Now we can see that the triple $(\tilde X,\tilde C,\tilde Z)$ satisfies the requirements of Lemma~\ref{l5.7} with $\dim(\tilde Z)<\dim(Z)=n$. So, by the inductive assumption, the convex set $\tilde C_0$ is $\tilde C$-infinitely hiding.

Since $E=V_{E\cap C_0}=V_{E\cap C_0}-V_{E\cap C_0}$, we can apply Lemma~\ref{l5.4} to conclude that the open convex  set $C_0$ is $C$-infinitely hiding in $X$. But this contradicts our assumption.
\end{proof}

\begin{claim} $\dim(V_{C_0})\ge 2$.
\end{claim}

\begin{proof} Assume conversely that $\dim(V_{C_0})\le 1$. The equality $\dH(C_0,V_{C_0})=\infty$ implies that the open convex subset $C_0$ is unbounded in the finite-dimensional linear space $Z\cap C_0$ and consequently $V_{C_0}\ne 0$. Since $-V_{C_0}\cap V_{C_0}=\{0\}$, we conclude that $V_{C_0}=\bar\IR_+ e$ for some non-zero vector $e\in V_{C_0}$. Now consider the linear subspace $E=\IR e\subset Z$, the quotient Banach space $\tilde X=X/E$, the finite-dimensional linear subspace $\tilde Z=q(Z)$, the convex sets $\tilde C=q(C)$, and the open convex set $\tilde C_0=q(C_0)$, which is dense in $\tilde C$. Claim~\ref{cl5.8} guarantees that the set $\tilde C_0$ has trivial characteristic cone and hence is bounded in the finite-dimensional space $\tilde Z$. This implies that $\dH(C_0,V_{C_0})<\infty$, which is a desired contradiction.
\end{proof}

Since $C_0$ is not $C$-infinitely hiding, Lemma~\ref{l5.2}, guarantees that the characteristic cone $V_{C_0}$ is polyhedral in $Z$ and hence it can be written as a finite intersection of closed half-spaces $$V_{C_0}=\bigcap_{i=1}^k f_i^{-1}\big((-\infty,0]\big)$$determined by some linear functionals $f_1,\dots,f_k:Z\to\IR$. We shall assume that the number $k$ in this representation is the smallest possible.
%In this case for every positive $i\le k$ the intersection $V_{C_0}\cap f_i^{-1}(0)$ is a non-trivial (and hence unbounded) face of the cone $V_{C_0}$.

It follows from $\dH(C_0,V_{C_0})=\infty$ and Lemma~\ref{l2.6} that $\sup f_i(C_0)=\infty$ for some $i\le k$.

\begin{claim}\label{cl5.11} The face $f_i^{-1}(0)\cap V_{C_0}$ of $V_{C_0}$ contains a non-zero vector $e$.
\end{claim}

\begin{proof} By the minimality of $k$ the cone $V=\bigcap\{f^{-1}_j\big((-\infty,0]\big):1\le j\le k,\;\;j\ne i\}$ is strictly larger than $V_{C_0}$ and hence contains a point $x\in V\setminus V_{C_0}$. For this point $x$ we get $f_j(x)\le 0$ for all $j\ne i$, and $f_i(x)>0$.

Since $\dim(V_{C_0})\ge 2$, there exists a vector $y\in V_{C_0}\setminus\IR x$. Such choice of $x$ guarantees that $0\notin [x,y]$. Since $f_i(y)\le 0$ and $f_i(x)>0$, there is a point $e\in [x,y]$ with $f_i(e)=0$.
For every $j\ne i$, the inequalities $f_j(x)\le 0$ and $f_j(y)\le 0$ imply $f_j(e)\le 0$. Conseqeuntly, $e$ is a required non-zero vector in $f_i^{-1}(0)\cap V_{C_0}$.
\end{proof}

Claim~\ref{cl5.11} yields a non-zero vector $e\in f_i^{-1}(0)\cap V_{C_0}$. Consider the 1-dimensional linear subspace $E=\IR e$ of $X$ and let $\tilde X=X/E$ be the quotient Banach space. Observe that $\tilde X$  contains the finite-dimensional linear subspace $\tilde Z=Z/E$ of dimension $\dim(\tilde Z)=\dim(Z)-1<n$. Let $q:X\to\tilde X$ be the quotient operator, and $\tilde C_0=q(C_0)$, $\tilde C=q(C)$ be the images of the convex sets $C_0$ and $C$ in $\tilde X$. It follows from $E=q^{-1}(0)\subset Z$ that $\tilde Z\cap\tilde C=q(Z\cap C)$ and $\tilde C_0=q(C_0)$ coincides with the interior of the set $q(Z\cap C)=\tilde Z\cap\tilde C$ in $\tilde Z$. So, the triple $(\tilde X,\tilde Z,\tilde C)$ satisfies the assumptions of the lemma.

We claim that $\dH(\tilde C_0,V_{\tilde C_0})=\infty$.
Since $E\subset f_i^{-1}(0)$, there is a linear functional $\tilde f_i:\tilde Z\to \IR$ such that $f_i=\tilde f_i\circ q|Z$.

\begin{claim} $V_{\tilde C_0}\subset \tilde f_i^{-1}(-\infty,0]$.
\end{claim}

\begin{proof} Assume conversely that the characteristic cone $V_{\tilde C_0}$ contains some vector $w\in \tilde Z$ with $\tilde f_i(w)>0$. Then $\IR_+w\subset \tilde Z$. Pick any vector $v\in q^{-1}(w)\subset Z$ and consider the two-dimensional subspace $Z_2=\lin(\{v,e\})$ spanning the vectors $v$ and $e$. Observe that for every $t\in\IR$ we get  $f_i(v+te)=\tilde f_i(w)>0$, which implies that $(v+t\IR)\cap V_{Z_2\cap C}=\emptyset$. Then  $V_{Z_2\cap C}\subset \IR e-\bar\IR_+v$ and $q(V_{Z_2\cap C})\subset -\bar\IR_+w$. On the other hand, the projection $q(Z_2\cap C)$ contains the half-line $\IR_+ w$, which implies that $\dH(Z_2\cap C,V_{Z_2\cap C})=\infty$. But this contradicts Claim~\ref{cl5.8}.
\end{proof}

Taking into account that $\infty=\sup f_i(C_0)=\sup \tilde f_i(\tilde C_0)$, we conclude that $\dH(\tilde C_0,V_{\tilde C_0})=\infty$ and by the inductive assumption, the open convex set $\tilde C_0$ is $\tilde C$-infinitely hiding (as $\dim(\tilde Z)<\dim(Z)=n$).
Since $E=\IR_+e-\IR_+e=V_{E\cap C_0}-V_{E\cap C_0}$, we can apply Lemma~\ref{l5.4} to conclude that the open convex  set $C_0$ is $C$-infinitely hiding in $X$. This contradiction completes the proof of Lemma~\ref{l5.7}.
\end{proof}

Lemma~\ref{l5.7} admits the final (and main) lemma of this section.

\begin{lemma}\label{l5.8} A closed convex subset $C$ of a Banach space $X$ is infinitely hiding if  $\dH(A\cap C,A\cap V_C)=\infty$ for some finite-dimensional affine subspace $A\subset X$.
\end{lemma}

\begin{proof} It is well-known that $C\cap A$ has non-empty interior $C_0$ in its affine hull $\aff(C\cap A)$. Moreover, $C\cap A$ coincides with the closure $\bar C_0$ of $C_0$ in $A$.
Shifting the set $C$, if necessary, we can assume that $0\in C_0$. Then $Z=\aff(C_0)$ is a finite-dimensional linear subspace of $X$ such that $C\cap Z=C\cap A$ has non-empty interior $C_0$ which contains zero and is dense in $C\cap Z$. It follows that $\dH(C_0,V_{C_0})=\dH(Z\cap C,V_{Z\cap C})=\infty$ and hence $C_0$ is $C$-infinitely hiding by Lemma~\ref{l5.7}, and $C$ is infinitely hiding as $C_0\subset C$.
\end{proof}

\section{Approximating by positively hiding convex sets}\label{s6}

In this section we search for conditions guaranteeing that a closed convex subset $C$ of a Banach space can be approximated by positively hiding convex subsets of $X$.

At first we construct biorthogonal sequences, which are related to convex sets that have trivial characteristic cone.

We recall that a sequence of pairs $\{(x_n,x^*_n)\}_{n\in \w}\subset X\times X^*$ is {\em  biorthogonal} if $x^*_n(x_n)=1$ and $x^*_n(x_k)=0$ for all $n\ne k$, see \cite[1.1]{HMVZ}.

\begin{lemma}\label{l6.1}  Assume that a closed convex subset $C$ of an infinite-dimensional Banach space $X$ has trivial characteristic cone $V_C=\{0\}$. Then there exists a biorthogonal sequence $\{(x_n,x_n^*)\}_{n\in\w}\subset X\times V^*_C$ such that $\|x^*_n\|=1\le \|x_n\|<4$ for all $n\in\w$.
\end{lemma}

\begin{proof} Replacing $C$ by a shift of its closed neighborhood, we can assume that the convex subset $C$ of $X$ has non-empty interior $C_0$, which contains zero. After such replacement the cones $V_C$ and $V_C^*$ will not change, see Lemma~\ref{l2.4n}.

The biorthogonal sequence $\{(x_n,x_n^*)\}_{n\in\w}$ will be constructed by induction. We start by choosing an arbitrary functional $x^*_0\in V^*_C$ and a point  $x_0\in X$ with $1=\|x_0^*\|=x^*_0(x_0)\le \|x_0\|<4$.

Assume that for some $k\in\w$ a finite  biorthogonal sequence $\{(x_n,x^*_n)\}_{n<k}\subset X\times V^*_C$ has been constructed so that $1=\|x^*_n\|\le\|x_n\|<4$ for all $n< k$. Let $L^*$ be the linear hull of the finite set $\{x^*_0,\dots,x^*_{k-1}\}$ in the dual Banach space $X^*$. In the compact set $L^*_2=\{x^*\in L:\|x^*\|\le 2\}$, choose a finite subset $F^*_2\subset L^*_2$ such that for each functional $x^*\in L^*_2$ there is a functional $y^*\in F^*_2$ with $\|x^*-y^*\|<\frac18$. For every functional $f\in F^*_2$ choose a point $x_f\in X$ such that $\|x_f\|=1$ and $f(x_f)>\|f\|-\frac18$.

Let $E$ be the linear hull of the finite set $\{x_i:i<k\}\cup\{x_f:f\in F_2^*\}$ in $X$.
Let $\tilde X=X/E$ be the quotient Banach space, $q:X\to\tilde X$ be the quotient operator, and $\tilde C=q(C)$ be the quotient image of the convex set $C$ in $\tilde X$. Since the quotient operator $q$ is open, the image $\tilde C_0=q(C_0)$ coincides with the interior of $\tilde C$ in $\tilde X$. We claim that the set $\tilde C_0$ has trivial characteristic cone $V_{\tilde C_0}=\{0\}$. Assuming the converse, find a non-zero vector $\tilde v\in V_{\tilde C_0}$, choose any vector $v\in q^{-1}(\tilde v)$ and consider the finite-dimensional linear subspace $E_v=\lin(E\cup\{v\})$. The triviality of the characteristic cone $V_C=\{0\}$ implies that the intersection $E_v\cap C_0$ is bounded and then the image $q(E_v\cap C_0)=\IR \tilde v\cap \tilde C_0$ also is bounded. So, $\bar\IR_+\tilde v\not\subset C_0$, which contradicts $\tilde v\in V_{\tilde C_0}\setminus\{0\}$.

The triviality of the cone $V_{\tilde C_0}=\{0\}$ implies that $\tilde C\ne \tilde X$, so we can find a functional $\tilde x_k^*\in V^*_{\tilde C_0}$ with unit norm $\|\tilde x_k^*\|=1$. Now consider the functional $x_k^*=\tilde x_k^*\circ q$ and observe that $\|x_k^*\|=\|\tilde x^*_k\|=1$ and $x_k^*(x_i)=0$ for all $i<k$ and $x_k^*(x_f)=0$ for all $f\in F_2^*$. We claim that $\dist(x^*_k,L^*)>\frac14$. Assuming the converse, find a functional $l^*\in L^*$ with $\|x^*_k-l^*|\le \frac14$. Then $\|l^*\|\le\|x^*_k\|+\frac14$, so $l^*\in L_2^*$ and by the choice of the set $F_2^*$, we can find a functional $f\in F^*_2$ such that $\|l^*-f\|<\frac18$. Then $\|x^*_k-f\|\le\|x^*_k-l^*\|+\|l^*-f\|\le \frac14+\frac18=\frac38$, $\|f\|\ge\|x^*_k\|-\|x^*_k-f\|\ge 1-\frac38=\frac58$ and we obtain the contradiction: $$\frac48=\frac58-\frac18\le \|f\|-\frac18<f(x_f)=|x_k^*(x_f)-f(x_f)|\le\|x^*_k-f\|\cdot\|x_f\|\le\frac38.$$

This contradiction shows that $\dist(x^*_k,L^*)>\frac14$ and implies that the the closed $\frac14$-ball $B^*=\{x^*\in X^*:\|x^*-x^*_{k}\|\le \frac14\}$ centered at $x^*_{k}$ does not intersect $L^*$. By the Banach-Alaoglu Theorem this ball is compact in the weak$^*$ topology of $X^*$. Now the Hahn-Banach Theorem applied to the weak$^*$ topology of $X^*$ yields us a point $x_{k}\in X$ that separates $L^*$ and $B^*$ in the sense that $\sup_{x^*\in L^*}x^*(x_{k})< \inf_{x^*\in B^*}x^*(x_{k})$.
It follows from the linearity of $L^*$ that $x^*(x_{k})=0$ for all $x^*\in L^*$.
In particular, $x^*_i(x_{k})=0$ for all $i< k$. Multiplying $x_{k}$ by a suitable positive constant, we may additionally assume that $x^*_{k}(x_{k})=1$.
Then $\|x_{k}\|\ge 1$ because $x^*_{k}$ has unit norm. To finish the inductive step it suffices to check that $\|x_k\|<4$. Assuming the converse, find a functional $x^*\in X^*$ with unit norm such that $x^*(x_k)=\|x_k\|\ge 4$. Then the functional $y^*=x^*_k-\frac14x^*$ belongs to the ball $B^*$ and thus $y^*(x_k)>0$. On the other hand, $y^*(x_k)=x^*_k(x_k)-\frac14x^*(x_k)\le 1-\frac14\cdot4= 0$, which is a desired contradiction.
\end{proof}

\begin{lemma}\label{l6.2}  Assume that a closed convex subset $C$ of an infinite-dimensional Banach space $X$ has trivial characteristic cone $V_C=\{0\}$. Then for each $\e>0$ there is a positively hiding closed convex set $C_\e\subset X$ with $\dH(C_\e,C)\le \e$.
\end{lemma}

\begin{proof} By Lemma~\ref{l6.1}, there exists a biorthogonal sequence $\{(x_n,x_n^*)\}_{n\in\w}\subset X\times V^*_C$ such that $1=\|x_n^*\|\le \|x_n\|<4$ for all $n\in\w$.

Then for every $\e>0$, a positively hiding convex set $C_\e$ with $\dH(C_\e,C)\le\e$ can be defined by the formula:
$$C_\e=\big\{x\in \cl(C+\e\IB):\forall n\in\w\;\;x^*_n(x)\le \tfrac18\e+\sup x^*_n(C)\big\}$$where $\IB=\{x\in X:\|x\|<1\}$ stands for the open unit ball of the Banach space $X$.
It is clear that $C\subset C_\e\subset \cl(C+\e\IB)$, which implies that $\dH(\tilde C,C)\le\e$.
It remains to check that the set $\tilde C$ is positively hiding. This can be done as follows.

For every $n\in\w$ choose a point $c_n\in C$ with $x^*_n(c_n)>\sup x^*_n(C)-\tfrac1{16}\e$ and consider the point $a_n=c_n+\frac{\e}4 x_n$. We claim that $\dist(a_n,C_\e)\ge \frac1{16}\e$. Indeed, for any point $c\in C_\e$, we get $x^*_n(c)\le\sup x^*_n(C)+\frac18\e$ while
$$x^*_n(a_n)=\frac\e4 x^*_n(x_n)+x^*_n(c_n)>\frac\e4+\sup x^*_n(C)-\frac\e{16}=\frac{3\e}{16}+\sup x^*_n(C).$$ Consequently, $\|a_n^*-c\|=\|x^*_n\|\cdot\|a_n-c\|\ge x^*_n(a_n)-x^*(c)\ge \frac{3\e}{16}-\frac{\e}{8}=\frac\e{16}$.

So, the set $A=\{a_n\}_{n\in\w}$ lies on positive distance $\inf_{a\in A}\dist(a,C_\e)\ge\frac1{16}\e$ from $C_\e$.
To show that the set $A$ is infinite and hidden behind $C_\e$, it suffices to check that for any distinct numbers $n,m$ the midpoint $\frac12a_n+\frac12a_m$ of the
segment $[a_n,a_m]$ belongs to the convex set $C_\e$.
Taking into account that $a_n,a_m\subset \cl(C+\e\IB)$, we conclude that $\frac12a_n+\frac12a_m\in [a_n,a_m]\subset \cl(C+\e\IB)$.
The inclusion $\frac12a_n+\frac12a_m\in C_\e$ will follow from definition of $C_\e$ as soon as we check that $x^*_k(\frac12a_n+\frac12a_m)\le\sup x^*_k(C)+\frac18\e$ for every $k\in\w$.

If $k\notin\{n,m\}$, then $x^*_k(x_n)=x^*_k(x_m)=0$ and hence
$$x^*_k(\tfrac12a_n+\tfrac12a_m)=x^*_k(\tfrac12c_n+\tfrac12c_m)\le\sup x^*_k(C).$$
If $k=n$, then
$$x^*_k(\tfrac12a_n+\tfrac12a_m)=x^*_n(\tfrac12c_n+\tfrac12c_m)+\tfrac18\e x^*_n(x_n)\le\sup x^*_n(C)+\tfrac18\e.$$By analogy we can treat the case $k=m$.
\end{proof}

\begin{lemma}\label{l6.3}  Assume that for a closed convex subset $C$ of a Banach space $X$ the closed linear subspace $Z=\cl(V_C-V_C)$  has infinite codimension in $X$. Then for each $\e>0$ there is a positively hiding convex set $C_\e\subset X$ with $\dH(C_\e,C)\le \e$.
\end{lemma}

\begin{proof} Using Lemma~\ref{l3.3} (by analogy with the proof of Lemma~\ref{l3.4}), we can reduce the proof to the case $-V_C\cap V_C=\{0\}$. So, from now on we assume that $-V_C\cap V_C=\{0\}$. Replacing $C$ by a shift of its closed neighborhood, we can assume that $C$ has non-empty interior $C_0$ in $X$ and $0\in C_0$. If the characteristic cone $V_C$ is not polyhedral in its closed linear hull $Z=\cl(V_C-V_C)$, then by Lemma~\ref{l5.3}, the convex set $C$ is infinitely (and hence positively) hiding and hence we can put $C_\e=C$. So, we assume that the characteristic cone $V_C$ is polyhedral in $Z$.
Since $-V_C\cap V_C=\{0\}$, the polyhedrality of $V_C$ implies that the closed linear space $Z=\cl(V_C-V_C)$ is finite-dimensional and coincides with $V_C-V_C$. Now consider the quotient Banach space $\tilde X=X/Z$, the quotient operator $q:X\to\tilde X$, and the convex set $\tilde C=q(C)$. Since the operator $q$ is open, the image $\tilde C_0=q(C_0)$ of the interior $C_0$ of $C$ coincides with the interior of $\tilde C$. Now consider the characteristic cone $V_{\tilde C_0}$ of the open convex set $\tilde C_0$.

If $V_{\tilde C_0}$ contains some non-zero vector $\tilde v$, then for any vector $v\in q^{-1}(v)$ and for the finite-dimensional linear subspace $E=\lin(Z\cup\{v\})$ the intersection
$E\cap C_0$ lies on infinite Hausdorff distance $\dH(E\cap C_0,V_{E\cap C_0})=\infty$ from its characteristic cone. This is so because $q(V_{E\cap C_0})\subset q(V_C)=\{0\}$ while $q(E\cap C_0)\supset\bar\IR_+\tilde v$. Then Lemma~\ref{l5.8} guarantees that the set $C$ is infinitely (and hence positively) hiding. In this case we can put $C_\e=C$.

So, it remains to consider the case $V_{\tilde C_0}=\{0\}$. In this case, Lemma~\ref{l6.2} yields a positively hiding closed convex set $\tilde C_\e\subset\tilde X$ with $\dH(\tilde C_\e,\tilde C_0)<\e$. Now consider the convex set $C_\e=(C+\e\IB)\cap q^{-1}(\tilde C_\e)$ and observe that $\dH(C_\e,C)\le\e$ and $q(C_\e)=\tilde C_\e$.

The set $\tilde C_\e$, being positively hiding, hides a countably infinite set $\tilde A\subset \tilde X\setminus\tilde C_\e$ on positive distance $\inf_{a\in\tilde A}\dist(a,\tilde C_\e)>0$ from $\tilde C_\e$. Taking into account that $Z=V_{Z\cap C_0}-V_{Z\cap C_0}$ and $q(C_\e)=\tilde C_\e$, with help of Lemma~\ref{l5.4}, we can find an infinite subset $A\subset q^{-1}(\tilde A)$ hidden behind the convex set $C_\e\subset C$. Since the quotient operator $q$ is not expanding, the set $A$ lies on positive distance
$$\inf_{a\in A}\dist(a,C)=\inf_{a\in A}\dist(a,C_0)\ge\inf_{\tilde a\in\tilde A}\dist(\tilde a,\tilde C_0)>0$$from the convex set $C_\e$. So, $C_\e$ is positively hiding.
\end{proof}

Our next approximation lemma will be used in the proof of the implication $(4)\Ra(2)$ of Theorem~\ref{main}.

\begin{lemma}\label{l6.4} Let $C$ be a closed convex set in a Banach space $X$. If $\dH(C,V_C)=\infty$, then for each $\e>0$ there is a positively hiding convex set $\tilde C\subset X$ with $\dH(\tilde C,C)\le \e$.
\end{lemma}

\begin{proof} If the closed linear subspace $V_C^\pm=\cl(V_C-V_C)$ has infinite codimension in $X$, then the existence of a positively hiding convex set $C_\e\subset X$ with $\dH(C_\e,C)<\e$ follows from Lemma~\ref{l6.3}. So, we assume that $V^\pm_C$ has finite codimension in $X$. If the characteristic cone $V_C$ is not polyhedral in $V_C^\pm$, then the set $C$ is infinitely (and positively) hiding by Lemma~\ref{l5.3}. In this case we can put $C_\e=C$. It remains to consider the case of polyhedral cone $V_C$ in $V_C^\pm$. Since $V_C^\pm$ has finite codimension in $X$, the characteristic cone $V_C$ is polyhedral in $X$ and hence the closed linear subspace $V^\mp_C=-V_C\cap V_C$ has finite codimension in $X$.

Then the quotient Banach space $Y=X/V^\mp_C$ is finite-dimensional. Let $q:X\to Y$ be the quotient operator. Lemma~\ref{l3.3} guarantees that $\dH(qC,V_{qC})=\dH(C,V_C)=\infty$ and then the set $qC$ is infinitely hiding in $Y$ by Lemma~\ref{l5.8}. By Lemma~\ref{l5.5}, the set $C$ is infinitely (and hence positively) hiding in $X$. Letting $C_\e=C$, we finish the proof of the lemma.
\end{proof}

\section{Proof of Theorem~\ref{main}}\label{pf:main}

To prove the first part of Theorem~\ref{main}, for every closed convex subset $C$ of a Banach space $X$  we should prove the equivalence of the following conditions:
\begin{enumerate}
\item $C$ is approximatively polyhedral;
\item the characteristic cone $V_C$ of $C$ is polyhedral in $X$ and $\dH(C,V_C)<\infty$;
\item  $\HH_C$ contains a polyhedral closed convex set;
\item  $\HH_C$ contains no positively hiding closed convex set;
\item the space $\HH_C$ is separable;
\item the space $\HH_C$ has density $\dens(\HH_C)<\mathfrak c$.
\end{enumerate}

It suffices to prove the implications $(1)\Ra(2)\Ra(3)\Ra(5)\Ra(6)\Ra(4)\Ra(2)\Ra(3)\Ra(1)$, among which $(2)\Ra(3)$ and $(5)\Ra(6)$ are trivial. The remaining implications can be established as follows.
\smallskip

To prove the implication $(1)\Ra(2)$, assume that the closed convex set $C$ is approximatively polyhedral and find a polyhedral convex set $P$ with $\dH(C,P)<\infty$. Lemma~\ref{l2.4n} implies that $V_C=V_P$. The polyhedral convex set $P$ can be written as a finite intersection $P=\bigcap_{i=1}^n f_i^{-1}\big((-\infty,a_i]\big)$ of closed hyperplanes determined by some functionals $f_1,\dots,f_n:X\to\IR$ and some real numbers $a_1,\dots,a_n\in\IR$. It is easy to check that $$V_C=V_P=\bigcap_{i=1}^n f_i^{-1}\big((-\infty,0]\big),$$
which means that the characteristic cone $V_C$ of $C$ is polyhedral. By Lemma~\ref{l2.6}, $\dH(P,V_C)=\dH(P,V_P)<\infty$. Consequently, $\dH(C,V_C)\le \dH(C,P)+\dH(P,V_C)<\infty$.
\smallskip

The implications $(3)\Ra(5)$ and $(3)\Ra(1)$ are proved in Lemma~\ref{l3.4} and $(6)\Ra(4)$ in Lemmas~\ref{l4.2}. The implication $(4)\Ra(2)$ follows from Lemmas~\ref{l5.3}, \ref{l6.3} and \ref{l6.4}.

Next, assuming the Banach space is finite-dimensional we will check that the conditions (1)--(6) are equivalent to:
\begin{enumerate}
\item[(7)] $C$ is not positively hiding;
\item[(8)] $C$ is not infinitely hiding.
\end{enumerate}

It suffices to check that $(4)\Ra(7)\Ra(8)\Ra (2)$. In fact, the first two implications $(4)\Ra(7)\Ra(8)$ are trivial, while $(8)\Ra(2)$ follows from Lemmas~\ref{l5.3} and \ref{l5.8}.

\section{Acknowledgements}

The authors would like to express their sincere thanks to Ostap Chervak and Sasha Ravsky for valuable and fruitful discussions, which resulted in finding a correct proof of Lemma~\ref{l5.8}.

%\newpage

\end{document}